\documentclass[11pt]{article}
\let\saveleq\le
\let\savegeq\ge
\def\addto#1{}
\usepackage[efaapaper]{faa_authors-utf8}

\newenvironment{proof}{{\bf \emph{Proof.} }}{\hfill $\Box$ \\} 

\let\le\saveleq
\let\ge\savegeq
\let\leq\saveleq
\let\geq\savegeq

           \usepackage[T2A]{fontenc}   
           \usepackage[utf8]{inputenc}
           \usepackage[english]{babel}
\usepackage{amsmath}
\usepackage{amsfonts}
\usepackage{amssymb}
\usepackage{mathrsfs}
\usepackage[pdftex,unicode]{hyperref}
\usepackage{array}
\usepackage{footnote}
\makesavenoteenv{tabular}
\makesavenoteenv{table}

\newtheorem{proposition}{Proposition}
\newtheorem{theorem}{Theorem}

\theoremstyle{definition}
\newtheorem{example}{Example}

\theoremstyle{remark}

\usepackage{hypbmsec}

\def\sq#1#2{(#1)_{#2}}
\def\sqn#1{\sq{#1}{n\in\om}}
\def\sqnn#1{\sqn{#1_n}}

\def\cl#1{\overline{#1}}
\def\clx#1#2{\overline{#2}^{\,#1}}

\def\sset#1{\{\,#1\,\}}

\def\set#1{\bbset#1\eeset{}}
\def\bbset#1:#2\eeset{\{\,#1 : #2\,\}}

\def\bbsett#1:#2\eesett{\{\, #1 : \text{#2}\,\}}

\def\ibbset#1:#2\ieeset{(#1)_{#2}}

\def\wt#1{\widetilde{#1}}
\def\restrB#1#2{\ensuremath{\left.#1\right|_{#2}}}

\def\restr#1#2{\restrB{#1}{#2}}
\def\term#1{{\it #1}}

\def\gd/{$G_\delta$}

\def\mmrus#1{\mu[#1]}
\def\ffmmrus#1{\mu^{\#}[#1]}
\def\mmus#1{\mu_{#1}}
\def\ffmus#1{\mu_{#1}^{\#}}

\def\mus#1/{$\mmus{#1}$}
\def\mrus#1/{$\mmrus{#1}$}
\def\fmus#1/{$\ffmus{#1}$}
\def\fmrus#1/{$\ffmmrus{#1}$}

\def\hmmrus#1{h\mu[#1]}
\def\hffmmrus#1{h\mu^{\#}[#1]}
\def\hmmus#1{h\mu_{#1}}
\def\hffmus#1{h\mu_{#1}^{\#}}

\def\hmus#1/{$\hmmus{#1}$}
\def\hmrus#1/{$\hmmrus{#1}$}
\def\hfmus#1/{$\hffmus{#1}$}
\def\hfmrus#1/{$\hffmmrus{#1}$}

\def\mMrus#1{\gM[#1]}
\def\mMus#1{\gM_{#1}}
\def\ffMus#1{\gM_{#1}^{\#}}
\def\ffMrus#1{\gM^{\#}[#1]}

\def\Mus#1/{$\mMus{#1}$}
\def\Mrus#1/{$\mMrus{#1}$}
\def\fMus#1/{$\ffMus{#1}$}
\def\fMrus#1/{$\ffMrus{#1}$}

\def\cscf#1{ $\mathrm{ ( SCF_{#1} ) }$ }

\def\sqi#1{SQ_{#1}}
\def\tsqi#1/{{\rm ($\sqi{#1}$)}}

\def\rarr{\Rightarrow}
\def\larr{\Leftarrow}

\def\si{\sigma}
\def\Si{\Sigma}
\def\nom{{n\in\om}}

\def\R{\mathbb R}
\def\N{\mathbb N}

\def\om{\omega}
\def\Om{\Omega}

\def\be{\beta}

\def\al{\alpha}
\def\ga{\gamma}
\def\la{\lambda}
\def\ph{\varphi}

\def\cQ{\mathcal{Q}}

\def\cG{\mathcal{G}}

\def\Int{\operatorname{Int}}

\def\bt{{\beta\,}}

\def\es{\varnothing}
\def\Tau{{\mathcal T}}

\def\nom{{n\in\om}}
\def\cM{{\mathcal M}}
\def\tp{\Tau}

\def\cP{{\mathcal P}}

\def\cF{{\mathcal F}}

\def\cR{{\mathcal R}}

\def\F{\Phi}
\def\te{\theta}
\def\D{\Delta}
\def\ep{\varepsilon}
\def\gm{{\Gamma}}

\def\nstep{\begin{itemize}\item[{\parbox[t][1em]{3em}{{\bf $n$th \\ step.}}}]}

\def\endnstep{\end{itemize}}

\def\ccite#1{\ {\rlap{\cite{#1}}}}
\def\mmm/{$\pmb{-}$}
\def\ppp/{$\pmb{+}$}
\def\nnn/#1/{{ #1}}
\def\iii#1{{\rm ($#1$)}}

\begin{document}

\frenchspacing

\title{Grothendieck's theorem on the precompactness of subsets functional spaces over pseudocompact spaces}

\date{December 1, 2023}
\datedor{December 1, 2023}
\dateprin{December 1, 2023}

\doi{10.4213/faaXXXX}

\author{Evgenii Reznichenko}

\address{Department of General Topology and Geometry, Mechanics and  Mathematics Faculty, M.~V.~Lomonosov Moscow State University, Leninskie Gory 1, Moscow, 199991 Russia}
\email{erezn@inbox.ru}

\markboth{Grothendieck's theorem for pseudocompact spaces}{E.~Reznichenko}

\maketitle

\begin{fulltext}

\begin{abstract}
Generalizations of the theorems of Eberlein and Grothendieck on the precompactness of subsets of function spaces are considered: if $X$ is a countably compact space and $C_p(X)$ is a space of continuous functions on $X$ in the topology of pointwise convergence, then any countably compact subspace of the space $C_p(X)$ is precompact, that is, it has a compact closure. The paper provides an overview of the results on this topic. It is proved that if a pseudocompact $X$ contains a dense Lindelof $\Si$-space, then pseudocompact subspaces of the space $C_p(X)$ are precompact. If $X$ is the product \v Cech complete spaces, then bounded subsets of the space $C_p(X)$ are precompact. Results on the continuity of separately continuous functions were also obtained. 
\end{abstract}


\section{Introduction}

Let $X$ be a Tychonoff space and $C_p(X)$ the space of continuous functions on $X$ in the topology of pointwise convergence.
For compact $X$, Eberlein \cite{Eberlein1947} showed that every relatively countably compact subset of $C_p(X)$ is precompact.
\footnote{In \cite{Eberlein1947} an equivalent statement was proved, part of the Eberlein–Shmulyan theorem \--- a relatively countably compact subset of a Banach space in the weak topology is weakly precompact.}
A.~Grothendieck \cite{grot1952} showed that this result remains true for countably compact $X$.
Later these results were generalized in different directions \cite{Ptak1954,Pryce1971,ps1974,Arhangelskii1976en,AsanovVelichko1981,Arhangelskii1984en,Arkhangelskii1992cpbook,Arhangelskii1997,ChobanKenderovMoors2014,AlperinOsipov2023}.

Let $X$ be a space and $Y\subset X$. We say that
\begin{itemize}
\item[($\cR_k$)] $Y$ \term{precompact}\footnote{We also say that $Y$ is \term{relatively compact} in $X$} in $X$ if $\cl Y$ is compact;
\item[($\cR_{cc}$)] $Y$ \term{relatively countably compact} in $X$ if every sequence $\sqnn x\subset Y$ has a accumulation point in $X$;
\item[($\cR_{pc}$)] $Y$ \term{relatively pseudocompact} in $X$ if there is a pseudocompact $Z$ such that $Y\subset Z\subset X$;
\item[($\cR_{b}$)] $Y$ is \term{bounded} in $X$ if every continuous function $f: X\to \R$ is bounded on $Y$.
\end{itemize}
A space $X$ is called a \term{\mus cc/-complete} (\term{\mus pc/-complete}, \term{\mus b/-complete}) space if for each $Y\subset X$ the following condition holds: if $Y$ is relatively countably compact (relatively pseudocompact, bounded) in $X$, then $Y$ is precompact in $X$. A space $X$ is called a \term{\fmus cc/-complete} (\term{\fmus pc/-complete}, \term{\fmus b/-complete}) space if $C_p(X)$ is \term{\mus cc/-complete} (\term{\mus pc/-complete}, \term{\mus b/-complete}).

Let us formulate the most important generalizations of the theorems of Eberlein and Grothendieck.\footnote{The definition of the spaces listed below can be found in Section \ref{sec:defs}}. The spaces from the left column are $\wt\mu$-complete for $\wt\mu\in\sset{\ffmus{cc},\ffmus{pc},\ffmus{b}}$ if the corresponding column has the sign \ppp/ and there is a counterexample if the sign is \mmm/.

\begin{table}[h!] \centering 
 \begin{tabular}{ c | m{15em} | m{2.4em}| m{2.4em} | m{2.4em} }
\nnn/N/&
class of spaces & \fmus cc/ & \fmus pc/ & \fmus b/
\\
\hline
\hline
\nnn/1/&
compact & \ppp/\footnote{Eberlein theorem \cite{Eberlein1947}} & \ppp/ & \ppp/\ccite{AsanovVelichko1981}
\\
\hline
\nnn/2/&
countably compact & \ppp/\footnote{Grothendieck's theorem \cite{grot1952}} & \ppp / & \ppp/\ccite{AsanovVelichko1981}
\\
\hline
\nnn/3/&
countably pracompact & \ppp/ & \ppp/ & \ppp/\ccite{Arhangelskii1984en}
\\
\hline
\nnn/4/&
pseudocompact & \ppp/ & \mmm/\footnote{The space from \cite{Shakhmatov1986} is an example, as V.V.~Tkachuk noted, see also \cite{ReznichenkoTkachenko2024}} & \mmm/
\\
\hline
\nnn/5/&
$\si$-compact & \ppp/ & \ppp/ & \mmm/\footnote{Proposition III.4.18 of \cite{Arkhangelskii1992cpbook}, this result is due to O.G. Okunev}
\\
\hline
\nnn/6/&
{contains a dense
$\si$-compact subspace} & \ppp/\footnote{In textbooks on functional analysis this statement called the Eberlein-Grothendieck theorem \cite{Voigt2020book}} & \ppp/ & \mmm/
\\
\hline
\nnn/7/&
{contains dense
$\si$-pracompact subspace} & \ppp/ & \ppp/\ccite{Arhangelskii1984en} & \mmm/
\\
\hline
\nnn/8/&
{contains dense
$\si$-bounded\-subspace} & \ppp/\ccite{Haydon1972} & \mmm/ & \mmm/
\\
\hline
\nnn/9/&
Lindelöf $p$-spaces & \ppp/ & \ppp/ & \ppp/\ccite{Arkhangelskii1992cpbook}
\\
\hline
\nnn/10/&
Lindelöf $\Si$-spaces & \ppp/\ccite{Arhangelskii1997} & \mmm/\ccite{Arhangelskii1997} & \mmm/
\\
\hline
\nnn/11/&
$k_\si$-flavoured spaces & \ppp/ & \ppp/\ccite{Arhangelskii1984en} & \mmm/
\\
\hline
\nnn/12/&
$p_\si$-flavoured spaces & \ppp/\ccite{Arhangelskii1984en} & \mmm/ & \mmm/
\\
\hline
\nnn/13/&
spaces of countable tightness & \ppp/\ccite{Pryce1971} & \ppp/ & \ppp/\ccite{AsanovVelichko1981}
\\
\hline
\nnn/14/&
$k$-spaces & \ppp/\ccite{Pryce1971} & \ppp/ & \ppp/\ccite{AsanovVelichko1981}
\end{tabular}
\caption{Generalizations of Grothendieck's theorem, I.}
\label{intro:table:1}
\end{table}

Pseudocompact \fmus pc/-complete spaces will be called \term{\fmus pc/-pseudo\-compact\-pact} spaces.
A pair of spaces $X$ and $Y$ form a \term{Grothendieck pair} \cite{Reznichenko1994} if any continuous the image of $X$ in $C_p(Y)$ is precompact.
A pseudocompact space $X$ is \fmus pc/-pseudocompact if and only if $Y$ and $X$ is a Grothendieck pair for any pseudocompact space $Y$ (Proposition \ref{p:mpc:1}).
We call $X$ a \term{Korovin} space \cite{ReznichenkoTkachenko2024}, if $(X,X)$ is a Grothendieck pair.
Any \fmus pc/-pseudocompact space is Korovin. In \cite{Reznichenko1994,ReznichenkoUspenskij1998,Reznichenko2022-2,Reznichenko2022,ReznichenkoTkachenko2024} the importance of Korovin spaces for topological algebra is shown.
We are mostly interested in \fmus pc/-pseudocompact spaces.

In Section \ref{sec:mpc} some classes of \fmus pc/-pseudocompact spaces are obtained, Theorems \ref{t:mpc:1}, \ref{t:mpc:3} and \ref{t:mpc:4}. The following theorem summarizes the results from the listed theorems.
\begin{theorem}
\label{t:intro:1}
Let $X$ be a pseudocompact space, $Y\subset X=\cl Y$, and let $Y$ satisfy any of the following conditions:
\iii 1
is a pointwise almost $q_D$-space;
\iii 2
is $\si$-$\be$-unfavorable;
\iii 3
is strongly Bouziad;
\iii 4
is a $\wt W$-space;
\iii 5
has a countable $\pi$-character;
\iii 6
is a Lindelöf $\Si$-space;
\iii 7
$Y$ is a continuous image of a $k_\si$-flavoured space $Z$ such that one of the following conditions holds:
\begin{enumerate}
\item[\iii a]
$Z$ is $\om$-stable and $e(Z^n)\leq\om$ for any $\nom$;
\item[\iii b]
$\mathrm{(MA+\lnot CH)}$
$e(Z^n)\leq\om$ for any $\nom$;
\item[\iii c]
$\mathrm{(PFA)}$
$e(Z)\leq\om$.
\end{enumerate}
Then $X$ is \fmus pc/-pseudocompact.
\end{theorem}
Theorem \ref{t:intro:1}(6) (Theorem \ref{t:mpc:4}) is one of the main results of the paper.

In Section \ref{sec:scfps} we study separately continuous functions, we are interested in when such functions are quasi-continuous (Proposition \ref{p:scfps:4}). We note Theorem \ref{t:scfps:2}, which strengthens Bouziad's theorem \cite[Theorem 2.3]{Bouziad1993}.
In Section \ref{sec:rbs} we consider placement properties, and obtain a new wide subclass of \fmus cc/-complete spaces (Theorem \ref{t:rbs:1}). We solve problems \cite[Problem III.4.25]{Arkhangelskii1992cpbook} and \cite[Problem 8.9]{Arhangelskii1984en}.
In Section \ref{sec:gav} we study the Asanov-Velichko game. Broad subclasses of \fmus cc/-complete, \fmus pc/-complete, and \fmus b/-complete spaces are obtained (Theorems \ref{t:gav:1}, \ref{t:gav:2}, and \ref{t:gav:3}). It is proved that the product of \v Cech-complete spaces is \fmus b/-complete (Proposition \ref{p:gav:7}).
We supplement table \ref{intro:table:1} with new results from this paper, see table \ref{intro:table:2}.

\begin{table}[h!]
\centering
\begin{tabular}{ c | m{15em} | m{2.4em}| m{2.4em} | m{2.4em} }
\nnn/N/&
class of spaces & \fmus cc/ & \fmus pc/ & \fmus b/
\\
\hline
\hline
\nnn/15/&
$b_\si$-flavoured spaces& \ppp/ & \mmm/ & \mmm/
\\
\hline
\nnn/16/&
$w_b$-$q_f$-spaces & \ppp/ & \mmm/ & \mmm/
\\
\hline
\nnn/17/&
$w_c$-$q_f$-spaces & \ppp/ & \ppp/ & \mmm/
\\
\hline
\nnn/18/&
$w_s$-$q_f$-spaces & \ppp/ & \ppp/ & \ppp/
\\
\hline
\nnn/19/&
almost $q_D$-spaces & \ppp/ & \ppp/ & \ppp/
\\
\hline
\nnn/20/&
product \hfill \linebreak of \v Cech-complete spaces & \ppp/ & \ppp/ & \ppp/
\end{tabular}
\caption{Generalizations of Grothendieck's theorem, II.}
\label{intro:table:2}
\end{table}

\section{Definitions and Notations}
\label{sec:defs}

Further, `space' means `Tychonoff space'.
For a set $X$, denote by $\cP(X)$ or $2^X$ the set of all subsets of $X$.

The first countable ordinal is denoted by $\om$ --- natural numbers together with $0$. $\N=\om\setminus\sset 0$ --- positive natural numbers. The natural number $n\in\om$ is also an ordinal, $n=\sset{0,1,...,n-1}$.

The terminology and notations correspond to the books \cite{EngelkingBookGT,Arkhangelskii1992cpbook,HandbookOfTopology1984}.

\subsection{Compactness-type properties}

A space $X$ is called \term{countably compact} if every sequence $\sqnn x\subset X$ has a accumulation point in $X$.
A space $X$ is called \term{countably pracompact} if there is a dense subspace $Y\subset X=\cl X$ that is relatively countably compact in $X$.
A space $X$ is called \term{pseudocompact} if every continuous function on $X$ is bounded.

A space $X$ is called \term{$\si$-compact} (\term{$\si$-countably compact}, \term{$\si$-countably pracompact}, \term{$\si$-countably pseudocompact}) if $X=\bigcup_\nom X_n$, where each $X_n$ is compact (countably compact, countably pracompact, pseudocompact). A subset $Y\subset X$ is called \term{$\si$-bounded} if $Y=\bigcup_\nom Y_n$, where each $Y_n$ is bounded in $X$. A space $X$ is called \term{weakly countably pracompact} if it contains a dense $\si$-countably pracompact subspace.

A space $X$ is a \term{Lindelöf $p$-space} if $X$ can be closedly embedded in the product of a compact space and a separable metrizable space.
A space $X$ is a \term{Lindelöf $\Si$-space\footnote{Such spaces are also called countably determined}} if $X$ is a continuous image of a Lindelöf $p$-space.

\subsection{Families of subspaces generating a topology}

Let $X$ be a space, $\cM$ a family of subspaces of $X$, and $\bigcup \cM=X$. We will say that the family $\cM$
\begin{itemize}
\item[\rm ($\cG_{t}$)] \term{generates} $X$ if the following condition is satisfied: $F\subset X$ is closed if and only if $F\cap M$ is closed in $M$ for each $M\in\cM$;
\item[\rm ($\cG_{f}$)] \term{functionally generates} $X$ if the following condition is satisfied: the function $f\in \R^X$ is continuous if and only if $\restr fM\in C_p(X|M)$ for any $M\in \cM$;
\item[\rm ($\cG_{sf}$)] \term{strongly functionally generates} $X$ if the following condition holds: a function $f\in \R^X$ is continuous if and only if $\restr fM\in C_p(M)$ for any $M\in \cM$.
\end{itemize}
Note, ($\cG_{t}$) $\rarr$ ($\cG_{sf}$) $\rarr$ ($\cG_{f}$).

A space $X$ has \term{countable tightness} (is a \term{$k$-space}, is a \term{sequential space}) if the family of countable (compact, metrizable compact) subspaces of $X$ generates $X$.
A space $X$ has \term{weak functional countable tightness}
\footnote{This property is also called \term{countable mini tightness} \cite{Arhangelskii1983} or \term{countable $\R$-tightness} \cite{Arkhangelskii1992cpbook}}
(is a \term{$k_\R$-space}) if the family of countable (compact) subspaces functionally generates $X$.
A space $X$ has \term{functional countable tightness} if the family of countable subspaces strongly functionally generates $X$.

Denote, $\cP_s(X)$ --- separable subsets of $X$, $\cP_c(X)$ --- weakly countably pracompact subsets of $X$,
$\cP_p(X)$ --- subsets of $X$ in which some $\si$-pseudocompact subset is dense,
$\cP_b(X)$ --- subsets of $X$ in which some $\si$-bounded subset is dense.

A space $X$ is functionally generated by a family $\cP_s(X)$ if and only if $X$ has weak functional countable tightness.
A space $X$ is called \term{$k_\si$-flavoured} \cite{Arhangelskii1997} if it is functionally generated by a family $\cP_c(X)$.
A space $X$ is called \term{$b_\si$-flavoured} if it is functionally generated by the family $\cP_b(X)$. In \cite[Definition 8.8]{Arhangelskii1984en} such spaces were called $b\si$-spaces.
A space $X$ is called \term{$p_\si$-flavoured} if it is functionally generated by the family $\cP_p(X)$.

\subsection{Cardinal invariants}

Let $(X,\tp)$ be a space. Define \term{Lindelöf number} $l(X)$, \term{extend} $e(X)$:
\begin{align*}
l(X)=\min \{\tau\,:&\, \text{for any open cover }\ga\text{ of the space }X
\\
&\text{ there exists a subcover }\mu\subset \ga, \text{ such that }|\mu|\leq\tau \}
\\
e(X)=\min \{\tau\,:&\, \text{if }D\subset X\text{ is a discrete closed subset,}
\\
&\text{ then }|D|\leq\tau \}
\end{align*}

\subsection{Function spaces in the topology of pointwise convergence}

Let $X$ be a space. The set of continuous functions on $X$ is denoted by $C(X)$ and $\R^X$ is the set of all functions on $X$ with the Tikhonov product topology. The space of continuous functions on $X$ with the pointwise convergence topology is denoted by $C_p(X)$, this topology is induced by the embedding $C_p(X)\subset \R^X$. For $Y\subset X$ we denote
\begin{align*}
\pi^X_Y&: \R^X\to\R^Y,\ f\mapsto \restr fY,
&
C_p(X|Y)&= \pi^X_Y (C_p(X)).
\end{align*}
The mapping $\pi^X_Y$ is a \term{restriction mapping} or \term{projection mapping} of function spaces, $C_p(X|Y)\subset C_p(Y)$ is the space of continuous functions on $Y$ that extend to continuous functions on $X$.
For the mapping of sets $\ph: X\to Y$, we denote
\[
\ph^{\#}: \R^Y\to \R^X,\ f\mapsto f\circ \ph
\]
\term{dual mapping} to $\ph$. If $\ph$ is continuous, then
\[
\ph^\#(C_p(f(X)|Y))\subset C_p(X)
\]
and the mapping $\ph^\#$ embeds $C_p(f(X)|Y)$ into $C_p(X)$.
A continuous mapping $\ph$ is called \term{$\R$-quotient} if $\ph$ is surjective and the following condition is satisfied: the function $f\in \R^Y$ is continuous if and only if the function $f\circ \ph$ is continuous.
A surjective continuous mapping $\ph$ is $\R$-quotient if and only if $\ph^\#(C_p(Y))$ is closed in $C_p(X)$ \cite{Arkhangelskii1992cpbook}.

If $Y\subset C_p(X)$, then \term{the diagonal product}
\[
\D(Y): X\to C_p(C_p(X)|Y)\subset C_p(Y),\ \D(Y)(x)(f)=f(x)
\]
is continuous \cite[Proposition 0.5.1]{Arkhangelskii1992cpbook} and the mapping
\[
\D(\D(Y)(X)): Y\to C_p(\D(Y)(X))
\]
is a topological embedding.

A space $X$ is called a \term{Preiss-Simon} space if for every non-closed $A\subset X$ and $x\in \cl A\setminus A$ there exists a sequence $\sqnn U$ of open sets in $A$ that converges to $x$ \cite{Arkhangelskii1992cpbook}. Clearly, a Preiss-Simon space is a Fréchet-Urysohn space. Eberlein compact space  are compact subsets of Banach spaces in the weak topology. Eberlein compact space  are characterized as compact subspaces of $C_p(X)$ for compact $X$. Eberlein compact spaces  are Preiss-Simon spaces \cite{ps1974}, \cite[Section IV.5]{Arkhangelskii1992cpbook}. Preiss-Simon compact space  are characterized by the propery that pseudocompact subspaces are compact.

\subsection{Topological games}

A topological game is a game whose definition depends on the parameters $(X_1,X_2,...,X_m)$. Among the parameters is a topological space, we will assume that the first parameter $X_1$ is a space.

The game is defined by the \term{game move rule} $R(X_1,X_2,...,X_m)$ and the \term{game winner determination rule} $V$.
We consider games with two players, in this paper these are players $\al$ and $\be$.
Topological games use infinite countable\footnote{Sometimes topological games are of uncountable length} sequential games.

Depending on the parameters $(X_1,X_2,...,X_m)$, in accordance with the game move rule $R$, we define
\begin{itemize}
\item[\rm ($M$)] \term{the set of all admissible moves} $M$.
The set
\[
M^{<\om}=\bigcup_\nom M_n
\]
is the set of possible \term{partial plays} of the game.
\item[($\F$)]
$\F: M^{<\om}\to 2^M$ is the function of determining the admissible moves for the next move.
\item[($p$)] $p: \om\to P=\sset{\al,\be}$ is the function of determining the order of moves, the $n$-th move is made by player $p(n)$.
\end{itemize}

The players make \term{moves} sequentially, on the $n$-th move, the player $\te=p(n)\in\sset{\al,\be}$, whose turn it is to move, makes a move and chooses some object $x_n\in \F(x_0,x_1,...,x_{n-1})$. After a countable number of moves, we obtain a \term{play} $\sqnn x$. In accordance with the rule $V$ for determining the winner of the game, the winner is determined by the play $\sqnn x\in M^\om$.

\term{The strategy} $s_\te$ of player $\te\in\sset{\al,\be}$ is a function
\[
s_\te: M_\te^*=\prod_{n\in p^{-1}(\te)} M^n \to M,
\]
for which $s_\te(t)\in \F(t)$ holds. For $n\in p^{-1}(\te)$, on the $n$-th move, player $\te$ chooses $x_n=s_\te(x_0,x_1,...,x_{n-1})$.


A strategy $s_\te$ is called \term{winning} if, for any strategy of the other player, after the game, the play $\sqnn x$ is obtained, which, in accordance with the rule $V$ for determining the winner of the game, is determined as winning for $\te$.

A game with the rules of the game $R$, the rules $V$ for determining the winner and the parameters $(X_1,X_2,...,X_m)$ will be denoted as
\[
\gm(R(X_1,X_2,...,X_m),V).
\]
We will call the game \term{$\te$-favorable} if player $\te$ has a winning strategy and \term{$\te$-unfavorable} if player $\te$ does not have a winning strategy.

In the description, the game is divided into \term{steps}, during which the players make their moves.

\subsection{Miscellaneous topological properties}

Let $X$ be a space and $x\in X$.

A point $x$ is called a point of \term{countable character} if $x$ has a countable neighborhood base. A point $x$ is called a point of \term{countable $\pi$-character} if $x$ has a countable $\pi$-base. A family $\cP$ of open subsets of $X$ is called a \term{$\pi$-base} in $x$ if for any neighborhood $U$ of $x$ there is $V\in \cP$ such that $V\subset U$.

A \term{condensation} is a bijective continuous map onto.

A mapping $f: X\to Y$ of topological spaces is called \term{quasi-continuous} if $\Int f^{-1}(U)$ is dense in $f^{-1}(U)$ for any non-empty open $U\subset Y$. A mapping $f: X\times Y \to Z$ of topological spaces is called \term{strongly quasicontinuous in $(x_*,y_*)\in X\times Y$} if for any neighborhood $W\subset Z$ of $f(x_*,y_*)$ and any neighborhood $U\times V$ of $(x_*,y_*)$ there exists a nonempty open $U'\subset U$ and a neighborhood $V'\subset V$ of $y_*$ such that $f(U'\times V')\subset W$. If $f$ is strongly quasicontinuous at all points $X\times Y$, then $f$ is called \term{strongly quasicontinuous}.

Let $D\subset X=\cl D$. A point $x$ is called a \term{$q$-point ($q_D$-point (with respect to $D$))} if there exists a sequence $\sqnn U$ of neighborhoods of $x$ such that if $x_n\in U_n$ ($x_n\in U_n\cap D$), then the sequence $\sqnn x$ accumulates to some point $X$.

Define the game from \cite{moors2013}.
Let $X$ be a space, $x_*\in X$, $D\subset X=\cl D$. Define the rules of the game $NP(X,x_*,D)$.
Put $U_{-1}=X$.

\nstep
Player $\be$ chooses $x_n\in U_{n-1}\cap D$. Player $\al$ chooses a neighborhood $U_n$ of point $x_*$.
\endnstep

After a countable number of moves, we get the play:
\[
(x_0,U_0,...,x_n,U_n,...).
\]

We define the rule for determining the winner:
\begin{itemize}
\item[($SQ$)]
the sequence $\sqnn x$ accumulates to some point.
\end{itemize}

A point $x_*$ is called an \term{nearly $q_D$-point} if the game $\gm(NP(X,x_*,D),SQ)$ is $\al$-favorable.
A space $X$ is called a \term{$q$-space} if any point $x\in X$ is a $q$-point.
A space $X$ is called a \term{$q_D$-space (nearly $q_D$-space)} if there exists a dense $D\subset X=\cl D$ such that any point $x\in X$ is a $q_D$-point (nearly $q_D$-point) relative to $D$.
A space $X$ is called a \term{pointwise $q_D$-space (pointwise nearly $q_D$-space)} if for any point $x\in X$ there exists a dense $D\subset X=\cl D$ such that $x$ is a $q_D$-point (nearly $q_D$-point) relative to $D$.

A space $\R^\R$ is an nearly $q_D$-space but does not contain $q_D$-points \cite[Remarks 2]{moors2013}.

\subsection{Note on Terminology}

This paper introduces unified names for the classes of spaces under study.

In \cite{Arhangelskii1997,Arhangelskii1998,Reznichenko2022-2,AlperinOsipov2023,ReznichenkoTkachenko2024}
\mus cc/-, \fmus cc/-, \hmus cc/-, \hfmus cc/-complete\- spaces are called $g$-spaces, weakly Grothendieck spaces, hereditarily $g$-spaces, Grothendieck spaces;
\mus pc/-, \fmus pc/-, \hmus pc/-, \hfmus pc/-complete spaces are called $pc$-spaces, weakly $pc$-Grothendieck spaces, $hpc$-spaces, $pc$-Grothendieck spaces;
\mus b/-, \fmus b/-, \hmus b/-complete spaces are called $og$-spaces, weakly $oc$-Grothendieck spaces, hereditarily $og$-spaces, $oc$-Grothendieck spaces.

\mus b/-Complete spaces are called $\mu$-spaces in \cite{AsanovVelichko1981,AlperinOsipov2023}, $\mu$-complete in \cite{ArhangelskiiChoban2011}.

\section{Quasicontinuity of separately continuous functions}\label{sec:scfps}

Let $X$ and $Y$ be spaces and $\F: X\times Y\to \R$ be a separately continuous function. Denote
\begin{align*}
\ph_X &: X \to C_p(Y),\ \ph_X(x)(y)=\F(x,y),
\\
\ph_Y &: Y \to C_p(X),\ \ph_Y(y)(x)=\F(x,y).
\end{align*}

For $\F$, consider the conditions:
\begin{itemize}
\item[\cscf{1}]
there exists a dense \gd/ subset $D\subset Y=\cl D$ such that $\F$ is continuous at every point $(x,y)\in X\times D$;
\item[\cscf{2}]
the function $\F$ is quasi-continuous;
\item[\cscf{3}]
the closure of $\ph_X(X)$ in $C_p(Y)$ is compact;
\item[\cscf{4}]
$\F$ extends to a separately continuous function on $\bt X\times \bt Y$;
\item[\cscf{5}]
the closure of $\ph_Y(Y)$ in $C_p(X)$ is compact;
\item[\cscf{6}]
there exists a dense \gd/ subset $E\subset X=\cl E$ such that $\F$ is continuous at every point $(x,y)\in E\times Y$.
\end{itemize}

Condition \cscf 6 for $\F$ is called the \term{Namioka property}. If any separately continuous function on $X\times Y$ has the Namioka property, then $X$ and $Y$ are said to have the \term{Namioka property}.
There is a natural bijection between separately continuous functions on $X\times Y$ and continuous maps from $X$ to $C_p(Y)$: $\F\mapsto \F_X$. Therefore $X$ and $Y$ form a Grothendieck pair if and only if for every separately continuous function $\F$ on $X\times Y$, \cscf 3 holds.

If $X$ and $Y$ are pseudocompact spaces, then for any separately continuous function $\F$, conditions \cscf 1--\cscf 6 are equivalent \cite[Theorem 1]{Reznichenko2022-2}. Therefore, the following statement holds.

\begin{theorem} \label{t:scfps:1}
Let $X$ and $Y$ be pseudocompact spaces. Then the following conditions are equivalent:
\begin{enumerate}
\item
for $X$ and $Y$, the Namioka property holds;
\item
for $Y$ and $X$, the Namioka property holds;
\item
$X$ and $Y$ are a Grothendieck pair;
\item
$Y$ and $X$ are a Grothendieck pair;
\item
for any separately continuous function $\F: X\times Y\to \R$ any of the equivalent conditions holds:
\begin{enumerate}
\item
the function $\F$ is quasi-continuous;
\item
$\F$ extends to a separately continuous function on $\bt X\times \bt Y$.
\end{enumerate}
\end{enumerate}
\end{theorem}

\begin{proposition}\label{p:scfps:1}
Let $X$ and $Y$ be spaces and $\F: X\times Y\to \R$ be a separately continuous function.
If any of the following conditions holds, then $\cF$ is a quasi-continuous function:
\begin{enumerate}
\item
the Namioka condition \cscf{1};
\item
there exists a dense subset $D\subset Y=\cl D$ such that one of the following conditions holds:
\begin{enumerate}
\item
the function $\restr \F{X\times D}$ is quasi-continuous;
\item
the function $\F$ is quasi-continuous at points $X\times D$;
\item
the function $\F$ is strongly quasi-continuous at points $X\times D$;
\item
for $y\in D$, the set $\set{x\in X : \F\text{ is continuous at point }(x,y)}$ is dense in $X$;
\item
the function $\F$ is continuous at points $X\times D$.
\end{enumerate}
\end{enumerate}
\end{proposition}
\begin{proof}
Obviously, (1) $\rarr$ (e) $\rarr$ (d) and (c) $\rarr$ (b) $\rarr$ (a). From (a) follows that the function $\F$ is  quasi-continuous  \cite[Proposition 3]{Reznichenko2022-2}.

Let us check (d) $\rarr$ (c). Let $(x,y)\in X\times D$, $U\times V$ be a neighborhood of the point $(x,y)$, $O\subset \R$ be an open neighborhood of the point $\F(x,y)$. Since
\[
M=\set{x'\in X : \F\text{ is continuous at }(x',y)}
\]
is dense in $X$ and $\F$ is separately continuous, then there exists $x'\in U\cap M$ such that $\F(x',y)\in O$. Since $\F$ is continuous at $(x',y)$, there exists a neighborhood $U'\times V'$ of $(x',y)$ such that $\F(U'\times V')\subset O$ and $U'\times V'\subset U\times V$.
\end{proof}

We define the Christensen-Saint-Raymond game \cite{christensen1981,Saint-Raymond1983} and its modification from \cite{BarecheBouziad2010}.

Let $X$ be a space. We define the rules of the game $Chr(X)$. Two players, $\al$ and $\be$, play.
Let $U_{-1}=X$.

\nstep
Player $\be$ chooses a non-empty open $V_n\subset U_{n-1}$. Player $\al$ chooses a non-empty open $U_n\subset V_n$ and $x_n\in X$.
\endnstep

After a countable number of moves, we obtain the play:
\[
(V_0,U_0,x_0,...,V_n,U_n,x_n,...).
\]

Let us define the rule for determining the winner.
\begin{itemize}

\item[($OP_p$)]
player $\al$ wins if $\bigcap_\nom V_n\cap \cl{\set{x_n:\nom}}\neq\es$;
\item[($OP_f$)]
player $\al$ wins if the sets $\bigcap_\nom V_n$ and $\set{x_n:\nom}$ are not functionally separated.
\end{itemize}

A space $X$ is called \term{$\si$-$\be$-unfavorable ($\si_f$-$\be$-unfavorable\footnote{In \cite{BarecheBouziad2010} such spaces are called $\si_{C(X)}$-$\be$-unfavorable})} if the game $\gm(Chr(X),OP_p)$ ($\gm(Chr(X),OP_f)$) is $\be$-unfavorable.

The following theorem strengthens \cite[Theorem 2.3]{Bouziad1993}.

\begin{theorem}\label{t:scfps:2}
Let $X$ be a $\si_f$-$\be$-unfavorable space, $Y$ be spaces, and $y_*\in Y$ be an nearly $q_D$-point. Then the function $\F$ is strongly quasi-continuous at any point $(x_*,y_*)$ for all $x_*\in X$.
\end{theorem}
\begin{proof}
Let $D\subset Y=\cl D$ and $y_*$ be an nearly $q_D$-point with respect to $D$. Let $s$ be a winning strategy in the game $\cG=\gm(NP(Y,y_*,D),SQ)$,
\begin{equation}\label{eq:scfps:1}
W_n=s(y_0,W_0,...,y_{n-1},W_{n-1},y_n)
\end{equation}
for play $\sqn{y_n,W_n}$ of $\cG$.

Let $x_*\in X$. Suppose to the contrary, the function $\F$ is not strongly quasi-continuous at $(x_*,y_*)$. Let $f(x,y)=\F(x,y)-\F(x,y_*)$ for $(x,y)\in X\times Y$. Then the function $f$ is separately continuous, $f(x,y_*)=0$ for $x\in X$, and $f$ is not strongly quasi-continuous at $(x_*,y_*)$, that is, there exists $\ep>0$ and a neighborhood $U_*\times W_*$ of $(x_*,y_*)$, such that $U_*\times \sset{y_*}\subset \cl{f^{-1}(\R\setminus (-3\ep,+3\ep))}$. Let us put 
\[ 
Q_0=f^{-1}((-\ep,+\ep))\text{ and }Q_1=f^{-1}(\R\setminus [-2\ep,+2\ep])\cap (U_*\times W_*).
\] 
Then $X\times\sset{y_*}\subset Q_0$ and $U_*\times \sset{y_*}\subset \cl{Q_1}$. Let $\pi$ be the projection of $X\times Y$ onto $Y$.

Let us define a strategy for $\be$ in the game $\cQ=\gm(Chr(X),OP_f)$.
Let $U_{-1}=U_*$, $S_{-1}=W_*$ and $x_{-1}=x_*$.

When constructing at the $n$-th step, we will construct $(V_n,U_n,x_n)$ and
we will also construct $(y_n,S_n,W_n)$, where $y_n\in Y$, $S_n$ and $W_n$ are open non-empty neighborhoods of the point $y_*$.
In this case, the following condition will be satisfied:
\begin{itemize}
\item[($\star$)]
$y_n\in S_n\cap D \subset \cl{S_n}\subset S_{n-1} \cap W_{n-1}$,
$\sset{x_n}\times S_n\subset Q_0$ and $V_n \times \sset{y_n} \subset Q_1$.
\end{itemize}

\nstep
Let $S_n$ be a neighborhood of $y_*$ such that $\cl{S_n}\subset S_{n-1} \cap W_{n-1}$
and $\sset{x_n}\times S_n\subset Q_0$.
The set $S'=\pi(Q_1\cap (U_{n-1}\times S_n))$ is open and dense in $S_n$. Take $y_n\in Q\cap S'\subset W_n$.
Define $W_n$ by formula (\ref{eq:scfps:1}).
Take a non-empty open $V_n\subset U_{n-1}$ such that $V_n \times \sset{y_n} \subset Q_1$.
Player $\al$ chooses a non-empty open $U_n\subset V_n$ and $x_n\in X$.
\endnstep

Since the game $\cQ$ is $\be$-unfavorable, then for some game $\sqn{V_n,U_n,x_n}$ player $\al$ has won, i.e. the condition ($OP_f$) is satisfied, i.e. the set $R=\set{x_n: \nom}$ is not  functionally separated from $G=\bigcap_\nom V_n$.

Since $s$ is a winning strategy in the game $\cG$, then the sequence $\sqnn y$ accumulates to some point $y'\in Y$.
Let $g(x)=f(x,y')$ for $x\in X$.
From ($\star$) it follows that $y'\in S_n$, $(x_n,y')\in \sset{x_n}\times S_n\subset Q_0$ and $g(x_n)=f(x_n,y')\in(-\ep,+\ep)$ for $\nom$. Therefore, $g(R)\subset (-\ep,+\ep)$. From ($\star$) it follows that $G\times \sset{y_n}\subset V_n\times \sset{y_n}\subset Q_1$ and, therefore, $f(G\times \sset{y_n})\cap [-2\ep,+2\ep]=\es$ for $\nom$. Since $y'$ is a accumulation point of $\sqnn y$, then $g(G)\cap (-2\ep,+2\ep)=\es$. We obtain that the function $g$ functionally separates $R$ and $G$. A contradiction.
\end{proof}

The following proposition lists results on separately continuous functions.
Missing definitions can be found in the corresponding papers.

\begin{proposition}\label{p:scfps:2}
Let $X$ and $Y$ be spaces, $y\in Y$ and $\F: X\times Y\to \R$ be a separately continuous function.
\begin{enumerate}
\item
If $X$ is a strongly Bouziad space and $y$ is a $q^*_D$-point for some dense $D\subset Y$, then $\F$ is strongly quasicontinuous at every point of $X\times\sset y$ \cite[Lemma 4]{Moors2013-2}.
\item
If $X$ contains a dense $\wt W$-space and $Y$ is Baire, then $\F$ is quasi-continuous \cite[Proposition 4]{Reznichenko2022-2}.
\item
If $X$ contains a dense set of points with countable $\pi$-character and $Y$ is Baire, then $\F$ is quasi-continuous \cite[Corollary 1]{Reznichenko2022-2}.
\end{enumerate}
\end{proposition}

\begin{proposition}\label{p:scfps:3}
If $X$ is $\si$-$\be$-unfavorable and $Y$ is pseudocompact, then $X$ and $Y$ have the Namioka property.
\end{proposition}
\begin{proof}
Let $\F: X\times Y\to \R$ be a separately continuous function,
\[
\ph: Y\to C_p(X),\ \ph(x)(y)=\F(x,y).
\]
Since $\ph$ is continuous, $\ph(Y)$ is pseudocompact and hence bounded in $C_p(X)$.
It follows from \cite[Theorem 3.4]{Troallic2000} that the set
\[
M=\set{x\in X: \ph(Y)\text{ is equicontinuous at }x}
\]
of type \gd/ and dense in $X$. Then $\F$ is continuous at the points of $M\times Y$.
\end{proof}


The following assertion follows from Propositions \ref{p:scfps:1}, \ref{p:scfps:2}, \ref{p:scfps:3} and Theorem \ref{t:scfps:2}.

\begin{proposition}\label{p:scfps:4}
Let $\wt X$ and $\wt Y$ be spaces, $X\subset \wt X=\cl X$, $Y\subset \wt Y=\cl Y$ and
let $X$ and $Y$ satisfy any of the following conditions:
\begin{enumerate}
\item
$X$ is $\si_f$-$\be$-unfavorable and $X$ is a pointwise nearly $q_D$-space.
\item
$X$ is $\si$-$\be$-unfavorable and $Y$ is pseudocompact;
\item
$X$ is strongly Bouziad and the set of $q^*_D$-points is dense in $Y$;
\item
$X$ is $\wt W$-space and $Y$ is Baire;
\item
$X$ with countable $\pi$-character and $Y$ is Baire.
\end{enumerate}
Then any separately continuous function $\F: \wt X\times \wt Y\to \R$ is quasi-continuous.
\end{proposition}

\section{\fmus pc/-pseudocompact spaces}\label{sec:mpc}

A space $X$ is called a \term{$pf$-space} if every pseudocompact $Y\subset X$ contains a point with countable character \cite{Reznichenko2022-2}. We call $X$ a $pf^\#$-space if $C_p(X)$ is a $pf$-space.

\begin{proposition}[{\cite[Theorem 3]{Reznichenko2022-2}}]\label{p:mpc:1}
A pseudocompact space $X$ is \fmus pc/-pseudocompact if and only if one of the equivalent conditions holds:
\begin{enumerate}
\item
$X$ is a $pf^\#$-space;
\item
for any pseudocompact $Y\subset C_p(X)$ one of the equivalent conditions holds:
\begin{enumerate}
\item
$\cl Y$ is an Eberlein compactum;
\item
$Y$ is \fmus pc/-pseudocompact;
\item
the set $\{\,f\in Y\,:$ the restriction of the pointwise and uniform convergence topologies coincide in $f\,\}$ is dense in $Y$;
\item
the set $\{\,f\in Y\,:\,f$ is a point of countable character in $Y\,\}$ is dense in $Y$;
\item
the set $\{\,f\in Y\,:\,f$ is a point of countable $\pi$-character in $Y\,\}$ is dense in $Y$.
\end{enumerate}
\item
for any pseudocompact $Y$ one of the equivalent conditions holds:
\begin{enumerate}
\item
$X$ and $Y$ are a Grothendieck pair;
\item
$Y$ and $X$ are a Grothendieck pair;
\item
for $X$ and $Y$ the Namioka condition holds;
\item
for $Y$ and $X$ the Namioka condition holds.
\end{enumerate}
\end{enumerate}
\end{proposition}

The following assertion follows from Propositions \ref{p:scfps:4} and \ref{p:mpc:1} and Theorem \ref{t:scfps:1}.

\begin{theorem}\label{t:mpc:1}
Let $X$ be a pseudocompact space, $Y\subset X=\cl Y$, and let any of the following conditions hold:
\begin{enumerate}
\item
$Y$ is a pointwise nearly $q_D$-space;
\item
$Y$ is $\si$-$\be$-unfavorable;
\item
$Y$ is strongly Bouziad;
\item
$Y$ is $\wt W$-space;
\item
$Y$ has a countable $\pi$-character.
\end{enumerate}
Then $X$ is \fmus pc/-pseudocompact.
\end{theorem}
\begin{proof}
A pseudocompact space is $\si_f$-$\be$-unfavorable \cite{BarecheBouziad2010}.
Then from Propositions \ref{p:scfps:4} and \ref{p:mpc:1} and Theorem \ref{t:scfps:1} we obtain the theorem.
\end{proof}

A space $X$ is called a \term{$cf$-space} if any compact $Y\subset X$ contains a point with countable character. We call $X$ a $cf^\#$-space if $C_p(X)$ is a $cf$-space.

The following statement is well known.
\begin{proposition}\label{p:mpc:2-3}
Let $X$ be a pseudocompact space and $x\in X$. A point $x$ is a point of countable character in $X$ if and only if $x$ is a point of type \gd/.
\end{proposition}
\begin{proof}
Let $\sqnn U$ be a sequence of open subsets of $X$ such that $\bigcap_\nom U_n=\sset x$. There is a sequence $\sqnn V$ of open subsets of $X$ such that
\[
x\in V_{n+1}\subset \cl{V_{n+1}} \subset V_n\subset U_n
\]
for $\nom$.
Show that $\sqnn V$ is a base at $x$. Let $U\subset X$ be an open neighborhood of $x$. Show that $V_n\subset U$ for some $\nom$. Assume the contrary. There is an open neighborhood $V$ of $x$ such that $\cl V\subset U$. Put $W_n=V_n\setminus \cl{V}$. Then $W_n\neq\es$. Since $X$ is pseudocompact, the sequence $\sqnn W$ accumulates to some point $y\in X$. Since $W_n\subset V_n$, then $y\in \bigcap_\nom\cl{V_n}=\sset x$, that is, $x=y$. Since $x\in V$ and $W_n\cap V=\es$, we obtain a contradiction with the fact that $\sqnn W$ accumulates to $x$.

\end{proof}
It follows from this proposition that $X$ is a $pf$-space ($cf$-space) if and only if every pseudocompact (compact) $P\subset X$ has a point of type \gd/.

\begin{proposition}\label{p:mpc:2-2}
Let $X$ be a $pf$-space, $Y$ a space, and $f: Y\to X$ a continuous injective map. Then $Y$ is a $pf$-space.
\end{proposition}
\begin{proof}
Let $P\subset Y$ be pseudocompact. Then $\ph(Y)$ is pseudocompact and since $X$ is a $pf$-space, some point $x\in \ph(Y)$ is a point of type \gd/. Then $\ph^{-1}(x)$ is a point of type \gd/ in $P$.
\end{proof}

\begin{proposition}\label{p:mpc:2-1}
If $X$ is a \mus pc/-complete $cf$-space, then $X$ is a $pf$-space.
\end{proposition}
\begin{proof}
Let $P\subset X$ be pseudocompact. Then $K=\cl P$ is compact and some point $x\in K$ has a countable base $\sqnn U$ in $K$.
Let $V_n=U_n\cap P$ for $\nom$. The sequence $\sqnn V$ has a unique accumulation point $x$ in $K$. Since $P$ is pseudocompact, $x\in P$.
\end{proof}

\begin{proposition}\label{p:mpc:2}
A space $X$ is a $pf^\#$-space if any of the following conditions holds.
\begin{enumerate}
\item
$X$ is a continuous image of a $pf^\#$-space;
\item
$X$ has a dense $pf^\#$-space;
\item
$X$ is a \fmus pc/-complete $cf^\#$-space.
\end{enumerate}
\end{proposition}
\begin{proof}
(1) Let $\ph: Y\to X$ be a continuous surjective map and $Y$ be a $pf^\#$-space.
Then $\ph^\#(C_p(X))\subset C_p(Y)$, $C_p(X)$ is embedded in $C_p(Y)$. Obviously, a subspace of a $pf$-space is a $pf$-space. Therefore, $C_p(X)$ is a $pf$-space.

(2) Let $Y\subset X=\cl Y$ be a $pf^\#$-space. The mapping $\pi^X_Y: C_p(X)\to C_p(Y)$ is injective, so Proposition \ref{p:mpc:2-2} implies that $C_p(X)$ is a $pf$-space.

(3) follows from Proposition \ref{p:mpc:2-1}.
\end{proof}

\begin{theorem}\label{t:mpc:2}
Let $X$ be a pseudocompact space and let $Y\subset X=\cl Y$ be a continuous image of a \fmus pc/-complete $cf^\#$-space. Then $X$ is \fmus pc/-pseudocompact.
\end{theorem}
\begin{proof}
Let $Y$ be a continuous image of a \fmus pc/-complete $cf^\#$-space $Z$. Proposition \ref{p:mpc:2}(3) implies that $Z$ is a $pf^\#$-space. Proposition \ref{p:mpc:2}(1) implies that $Y$ is a $pf^\#$-space. Proposition \ref{p:mpc:2}(2) implies that $X$ is a $pf^\#$-space.
Proposition \ref{p:mpc:1}(1) implies that $X$ is \fmus pc/-pseudocompact.
\end{proof}

\begin{proposition}[{\cite[Theorem 1]{Arhangelskii1976en}, \cite[Theorem 3.7']{Arhangelskii1984en}}]\label{p:mpc:3-1}
In a compact $\om$-monolithic space of countable tightness there is a point with countable character.
\end{proposition}

\begin{proposition}\label{p:mpc:3}
A space $X$ is a $cf^\#$-space if any of the conditions listed below is satisfied.
\begin{enumerate}
\item
any compact subspace of $C_p(X)$ is $\om$-monolithic and has countable tightness;
\item
$X$ is $\om$-stable and $e(X^n)\leq\om$ for any $\nom$;
\item
$\mathrm{(MA+\lnot CH)}$
$e(X^n)\leq\om$ for any $\nom$;
\item
$\mathrm{(PFA)}$
$e(X)\leq\om$;
\item
$X$ is a Lindelöf $\Si$-space;
\item
$X$ is a Lindelöf $p$-space.
\end{enumerate}
\end{proposition}
\begin{proof}
(1) follows from Proposition \ref{p:mpc:3-1}.

Let $K\subset C_p(X)$ be compact.
Let $Z=\D(K)(X)$. Then $Z\subset C_p(K)$ is a continuous image of $X$ and $K$ is embedded in $C_p(Z)$ by the mapping $\D(Z): K\to C_p(Z)$. Since $Z$ is embedded in $C_p(K)$ and $K$ is compact, it follows from Baturov's theorem \cite[Theorem III.6.1]{Arkhangelskii1992cpbook}, \cite[Theorem 1]{Baturov1987} that $l(Z^n)\leq e(X^n)$ for all $\nom$.
Further we will prove that $K$ is $\om$-monolithic and has countable tightness, after which the conclusion of the proposition follows from item (1).

(2) Since $X$ is $\om$-stable, then $C_p(X)$ and, consequently, $K$ is $\om$-monolithic \cite[Theorem II.6.8]{Arkhangelskii1992cpbook}. Since $l(Z^n)\leq e(X^n)\leq \om$ for $\nom$, then, by the Arkhangelskii--Pytkeev theorem \cite[Theorem I.1.1]{Arkhangelskii1992cpbook}, $C_p(Z)$ and $K$ have countable tightness.

(3) It follows from \cite[Theorem 4.3]{Arhangelskii1997} that, under the assumption $\mathrm{(MA+\lnot CH)}$, $K$ is $\om$-monolithic and has countable tightness.

(4) Since $Z$ is Lindelöf and $K$ embeds in $C_p(Z)$, $\mathrm{(PFA)}$ implies \cite[Theorem IV.11.14]{Arkhangelskii1992cpbook} that $K$ has countable tightness. It follows from \cite[Theorem 1.9]{OkunevReznichenko2007} that $\mathrm{(PFA)}$ implies that $K$ is $\om$-monolithic.

(5) Since the Lindelöf $\Si$-space is stable \cite[Theorem I.6.21]{Arkhangelskii1992cpbook}, $K$ is monolithic. The tightness of $C_p(X)$ for a Lindelöf $\Si$-space is countable \cite[Corollary II.16]{Arkhangelskii1992cpbook}, therefore the tightness of $K$ is countable.

(6) follows from (5).
\end{proof}

\begin{proposition}[{\cite[Corollary 2.7]{Arhangelskii1997}, \cite[Theorem 8.1]{Arhangelskii1984en}}]\label{p:mpc:4}
Any $k_\si$-flavoured space is \fmus pc/-complete.
\end{proposition}

In particular, (locally) separable $k$-spaces, $k_R$-spaces, sequential spaces, spaces with (weak functional) countable tightness are \fmus pc/-complete.

\begin{theorem}\label{t:mpc:3}
Let $X$ be a pseudocompact space and a dense subspace of $X$ be a continuous image of a $k_\si$-flavoured space $Y$ such that one of the following conditions holds:
\begin{enumerate}
\item
$Y$ is $\om$-stable and $e(Y^n)\leq\om$ for any $\nom$;
\item
$\mathrm{(MA+\lnot CH)}$
$e(Y^n)\leq\om$ for any $\nom$;
\item
$\mathrm{(PFA)}$
$e(Y)\leq\om$.
\end{enumerate}
Then $X$ is \fmus pc/-pseudocompact.
\end{theorem}
\begin{proof}
Proposition \ref{p:mpc:3} implies that $Y$ is a $cf^\#$-space. Proposition \ref{p:mpc:4} implies that $Y$ is a \fmus pc/-complete space. Theorem \ref{t:mpc:2} implies that $X$ is \fmus pc/-pseudocompact.
\end{proof}

\begin{proposition}\label{p:mpc:5}
If $X$ has a dense Lindelöf $\Si$-space, then $X$ is a $pf^\#$-space.
\end{proposition}
\begin{proof}
Let $D\subset X=\cl D$ be a Lindelöf $\Si$-space.
Then $D$ is a continuous image of some Lindelöf $p$-space $Y$.
The space $Y$ is \fmus b/-complete \cite[Theorem 2.15]{Arhangelskii1997} and, therefore, \fmus pc/-complete.
Proposition \ref{p:mpc:3}(6) implies that $Y$ is a $cf^\#$-space.
Proposition \ref{p:mpc:2}(3) implies that $Y$ is a $pf^\#$-space.
Proposition \ref{p:mpc:2}(1) implies that $D$ is a $pf^\#$-space.
Proposition \ref{p:mpc:2}(2) implies that $X$ is a $pf^\#$-space. \end{proof}

\begin{theorem}\label{t:mpc:4}
Let $X$ be a pseudocompact space and $X$ have a dense Lindelöf $\Si$-space.
Then $X$ is \fmus pc/-pseudocompact.
\end{theorem}
\begin{proof}
Proposition \ref{p:mpc:5} implies that $X$ is a $pf^\#$-space.
Proposition \ref{p:mpc:1}(1) implies that $X$ is \fmus pc/-pseudocompact.
\end{proof}

\section{Properties of placement of a subset of a space}\label{sec:rbs}

By \term{placement property} we mean a property that a subspace $Y$ may have with respect to the entire space $X$. If such a property $\cR$ holds for a pair $Y$, $X$, then we say that
$Y$ is \term{$\cR$-placed} in $X$. We call the placement property $\cR$ \term{continuously invariant} if the following condition is satisfied: if $Y\subset X$ is $\cR$-placed in $X$, $f: X\to Z$ is a continuous mapping, then $f(Y)$ is $\cR$-placed in $Z$.
A placement property $\cR$ is called a \term{boundedness-type} property if it is continuously invariant and satisfies the condition: in a space with a countable base, the closure of every $\cR$-placed set is compact.
For a boundedness-type property $\cR$, we also say `$Y$ is \term{$\cR$-bounded} in $X$' instead of `$Y$ is $\cR$-placed in $X$'.

The four boundedness-type properties were defined in the introduction --- $\cR_k$, $\cR_{cc}$, $\cR_{pc}$, and $\cR_{b}$.
A space $X$ is called \term{\mrus \cR/-complete} if $\cl Y$ is compact for every $Y$ $\cR$-placed subset of $X$.
A space $X$ is called \term{\fmrus \cR/-complete} if $C_p(X)$ is \mrus \cR/-complete.
\mus cc/-complete, \mus pc/-complete, and \mus b/-complete spaces are \mrus \cR_{cc}/-complete, \mrus \cR_{pc}/-complete, and \mrus \cR_b/-complete spaces, respectively.
\fmus cc/-complete, \fmus pc/-complete, and \fmus b/-complete spaces are \fmrus \cR_{cc}/-complete, \fmrus \cR_{pc}/-complete, and \fmrus \cR_b/-complete spaces, respectively.

Denote hereditarily \mrus \cR/-complete spaces as \hmrus \cR/-complete, and spaces $X$ for which $C_p(X)$ \hmrus \cR/-complete, as \hmrus \cR/-complete. Accordingly, denote \hmrus \cR_b/-complete spaces as
\linebreak
\hmus b/-complete, \hfmrus \cR_b/-complete spaces as \hfmus b/-complete, etc.

\begin{proposition}[{\cite{Arkhangelskii1992cpbook}}]\label{p:rbs:1}
Compact \hmus cc/-complete spaces are exactly the Fréchet-Urysohn compacta.
Compact \hmus pc/-complete spaces are exactly the Preiss-Simon compacta.

A space $X$ is \hmus cc/-complete if and only if $X$ is \mus cc/-complete and every compact subset of $X$ is a Fréchet-Urysohn compactum.

A space $X$ is \hmus pc/-complete if and only if $X$ is \mus pc/-complete and every compact subset of $X$ is a Preiss-Simon compactum.

Eberlein compacta are Preiss-Simon compacta.
\end{proposition}

\begin{proposition}\label{p:rbs:2}
Let $\cR$ be a boundedness-type property, $X$ be a space, $Y$ be $\cR$-bounded in $X$, and $Z\subset C_p(X)$ be \fmrus \cR/-complete.
\begin{enumerate}
\item
If $Z$ is bounded in $C_p(X)$, then $\clx{C_p(Y|X)}{\pi^X_Y(Z)}$ is Eberlein compactum.
\item
If $Z$ is pseudocompact, then $\pi^X_Y(Z)$ is Eberlein compactum.
\end{enumerate}
\end{proposition}
\begin{proof}
(1) Let $\ph=\D(Z): X\to C_p(Z)$. Denote $X'=C_p(Z)$, $Y'=\ph(Y)$, $\psi=\D(X')$, and $Z'=\psi(Z)\subset C_p(X')$.
Put $\ph_Y=\restr\ph Y$. Then $\ph_Y^\#: C_p(Y')\to C_p(Y)$ is a topological embedding and
\[
\ph_Y^\#(\pi^{X'}_{Y'}(Z'))=\pi^{X}_{Y}(Z)\text{ and }\ph_Y^\#(C_p(Y'|X'))\subset C_p(Y|X).
\]
Since $\cR$ is continuously invariant, $Y'$ is $\cR$-bounded in $X'$. Since $Z$ is \fmrus \cR/-complete, $K=\cl{Y'}$ is compact. Since compacta are $C$-embedded in any ambient space, $C_p(K|X')=C_p(K)$. Since $C_p(K)$ is \mus b/-complete and $Z''=\pi^{X'}_{K}(Z')$ is bounded in $C_p(K)$, $\cl{Z''}$ is an Eberlein compactum. Then $T=\pi^K_{Y'}(\cl{Z''})$ is an Eberlein compactum, $T\subset C_p(Y'|X')$ and $T=\cl{\pi^{X'}_{Y'}(Z')}$. Then $\cl{\pi^X_Y(Z)}=\ph_Y^\#(T)\subset C_p(Y|X)$ is an Eberlein compactum.

(2) It follows from (1) that $K=\clx{C_p(Y|X)}{S}$ is an Eberlein compactum, where $S=\pi^X_Y(Z)$. Since $S$ is a dense pseudocompact subset of the Eberlein compactum $K$ and the Eberlein compacta are Preiss-Simon compacta, then $K=S$.
\end{proof}

\begin{proposition}\label{p:rbs:3}
Let $X$ be a space, $P$ be a $\si$-bounded subset of $X$, $P\subset Y \subset \cl P$, and $Z$ be a relatively countably compact subset of $C_p(X)$. Then the closure of $\pi^X_Y(Z)$ in $C_p(Y|X)$ is an Eberlein compactum.
\end{proposition}
\begin{proof}
Let $P=\bigcup_\nom P_n$, where $P_n$ is bounded in $X$.
Let $F=\cl Z$. Then $F$ is a countably pracompact space and, therefore, \fmus b/-complete \cite[Theorem 2.9]{Arhangelskii1984en}.
It follows from Proposition \ref{p:rbs:2}(2) that $\pi^X_{P_n}(F)$ is an Eberlein compactum.
Then $K=\clx{\R^P}{F'}$ is an Eberlein compactum, where $F'=\pi^X_P(F)$. Since Eberlein compacta are Preiss-Simon compacta, $K=F'\subset C_p(K|X)$. Since $P$ is dense in $Y$, the mapping $\pi^Y_P: C_p(Y|X)\to C_p(K|X)$ is a condensation. The space $Q=\pi^X_Y(F)$ is pseudocompact and condenses onto the Eberlein compactum $K$. Consequently \cite{Arkhangelskii1992cpbook}, $Q$ is homeomorphic to $K$ and is an Eberlein compactum. It remains to note that $Q$ is the closure of $\pi^X_Y(Z)$ in $C_p(Y|X)$.
\end{proof}

This proposition positively solves \cite[Problem III.4.25]{Arkhangelskii1992cpbook}.

\begin{proposition}\label{p:rbs:4}
Let $X$ be a space, $\la$ be a cover of $X$, $Z\subset C_p(X)$.
If any of the conditions listed below is satisfied, then $Z$ is precompact in $C_p(X)$:
\begin{enumerate}
\item
$\la$ functionally generates $X$ and $\pi^X_Y(Z)$ is precompact in $C_p(Y|X)$ for any $Y\in \la$;
\item
$\la$ strongly functionally generates $X$ and $\pi^X_Y(Z)$ is precompact in $C_p(Y)$ for any $Y\in \la$.
\end{enumerate}
\end{proposition}
\begin{proof}
Let $K=\clx{\R^X}Z$ and $f\in K$. The set $Z$ is pointwise bounded, so $K$ is compact.

(1) $\restr fY\in \pi^X_Y(K)\subset C_p(Y|K)$ for $Y\in \la$. Since $\la$ functionally generates $X$, we have $f\in C_p(X)$.

(2) $\restr fY\in \pi^X_Y(K)\subset C_p(Y)$ for $Y\in \la$. Since $\la$ strongly functionally generates $X$, then $f\in C_p(X)$.
\end{proof}

\begin{theorem}\label{t:rbs:1}
$b_\si$-flavoured spaces are \fmus cc/-complete.
\end{theorem}
\begin{proof}
Let $X$ be a $b_\si$-flavoured space and let $Z$ be relatively countably compact in $C_p(X)$.

It follows from Proposition \ref{p:rbs:3} that $\pi^X_Y(Z)$ is precompact in $C_p(Y|X)$ for $Y\in\cP_b(X)$. It follows from Proposition \ref{p:rbs:4} that $Z$ is precompact in $C_p(X)$.
\end{proof}

Theorem \ref{t:rbs:1} strengthens \cite[Theorem 8.3]{Arhangelskii1984en}: $p_\si$-flavoured spaces are \fmus cc/-complete.
In \cite[Problem 8.9]{Arhangelskii1984en} the problem was posed: find an example of a space that is not $b_\si$-flavoured. Note that a non-discrete space in which every countable subset is $C$-embedded is not a $b_\si$-flavoured space, so, for example, the one-point Lindelification of a discrete uncountable space is not $b_\si$-flavoured.
Theorem \ref{t:rbs:1} allows us to construct other examples of spaces that are not $b_\si$-flavoured.

\begin{example}\label{e:rbs:1}
Let $C$ be some separable countably compact non-compact space \cite{EngelkingBookGT,rezn2020}, $X=C_p(C)$. The space $X$ is submetrizable, with countable Suslin number, and, since $C$ is closedly embedded in $C_p(X)$, it is not \fmus cc/-complete. It follows from Theorem \ref{t:rbs:1} that $X$ is not $b_\si$-flavoured.
\end{example}

\section{Weakly $q$-spaces and the Asanov-Velichko game}
\label{sec:gav}

In \cite{AsanovVelichko1981} a wide class of weakly $q$-spaces, a subclass of \fmus b/-complete spaces, is defined using a topological game. In \cite{AlperinOsipov2023} a modification of the Asanov-Velichko game is used to study precompact subsets of function spaces. In this paper a further modification of the Asanov-Velichko game is considered.

Let $X$ be a space, $A\subset X$, $x\in \cl A\setminus A$, $\la$ be some cover of $X$. Let us define the rules of the game $AV(X,x,A,\la)$. Two players, $\al$ and $\be$, play.

\nstep
Player $\be$ chooses a neighborhood $V_n$ of the point $x$. Player $\al$ chooses $S_n\subset A$ such that $S_n\subset L$ for some $L\in\la$.
\endnstep

After a countable number of moves, we get the play:
\[
(V_0,S_0,...,V_n,S_n,...).
\]

Let's define the rules for determining the winner.
\begin{itemize}
\item[($I$)]
player $\al$ has won if $\bigcap_\nom V_n\cap \cl{\bigcup_\nom S_n}\neq\es$;
\item[($I_f$)]
player $\al$ has won if there exists $S\subset \bigcap_\nom V_n$ such that $S\subset L$ for some $L\in \la$ and the sets $S$ and $\bigcup_\nom S_n$ are not functionally separated.
\end{itemize}

We say that
\begin{itemize}
\item[($W$-$q$)] $X$ \term{$W$-$q$-space with respect to $\la$} if for every non-closed $A\subset X$ there is $x\in \cl A\setminus A$, such that the game $\gm(AV(X,x,A,\la),I)$ is $\al$-favorable;
\item[($w$-$q$)] $X$ \term{$w$-$q$-space with respect to $\la$} if for every non-closed $A\subset X$ there is $x\in \cl A\setminus A$, such that the game $\gm(AV(X,x,A,\la),I)$ is $\be$-unfavorable;
\item[($W$-$q_f$)] $X$ \term{$W$-$q_f$-space with respect to $\la$} if for any discontinuous function $f: X\to\R$ there exists a non-closed $A\subset X$ and $x\in \cl A\setminus A$, such that $f(x)\notin\cl{f(A)}$ and
the game $\gm(AV(X,x,A,\la),I_f)$ is $\al$-favorable;
\item[($w$-$q_f$)] $X$ \term{$w$-$q_f$-space with respect to $\la$} if for any discontinuous function $f: X\to\R$ there exists a non-closed $A\subset X$ and $x\in \cl A\setminus A$, such that $f(x)\notin\cl{f(A)}$ and
the game $\gm(AV(X,x,A,\la),I_f)$ is $\be$-unfavorable.
\end{itemize}

\begin{proposition}\label{p:gav:1}
Let $f: X\to\R$ be a discontinuous function. Then
\begin{enumerate}
\item
there exists a non-closed $A\subset X$ such that $f(x)\notin \cl {f(A)}$ for all $x\in \cl A\setminus A$;
\item
if $D\subset X=\cl D$, then there exists a non-closed $A\subset D$ such that $f(x)\notin \cl {f(A)}$ for some $x\in \cl A\setminus A$.

\end{enumerate}
\end{proposition}
\begin{proof}
(1) Since $f$ is discontinuous, $A=f^{-1}(F)$ is non-closed for some closed $F\subset\R$. Then $f(x)\not\in F \supset \cl {f(A)}$ for all $x\in \cl A\setminus A$.

(2) It follows from \cite[Lemma 2.8]{Arhangelskii1984en} that there exists $x'\in X$ for which the function $\restr f{D'}$ is discontinuous on $D'$, where $D'=D\cup\sset {x'}$. It follows from (1) that there exists a set $A'\subset D'$ that is not closed in $D'$, such that $f(x)\notin \cl {f(A')}$ for some $x\in D'\cap \cl{A'}\setminus A'$. Put $A=A'\cap D=A\setminus\sset{x'}$. Then $f(x)\notin \cl {f(A)}$ and $x\in \cl A\setminus A$. \end{proof}

In \cite{AsanovVelichko1981} for $x\in \cl A\setminus A$ the concept of \term{weakly $q$-point with respect to $A$} is introduced, which coincides with the condition: the game $\gm(AV(X,x,A,\cP_s(X)),I)$ is $\al$-favorable.

We will call the space $X$
\term{($w_s$-$q_f$-...,$W_s$-$q_f$-...,$w_s$-$q$-...) a $W_s$-$q$-space} if $X$ ($w$-$q_f$-...,$W$-$q_f$-...,$w$-$q$-...) is a $W$-$q$-space with respect to $\cP_s(X)$.
The classes of weakly $q$-spaces and $W_s$-$q$-spaces coincide.

Let $\cP_1(X)=\set{\sset x: x\in X}$.
Let $X$ be called
\linebreak
\term{($w_1$-$q_f$-...,$W_1$-$q_f$-...,$w_1$-$q$-...) a $W_1$-$q$-space} if $X$ ($w$-$q_f$-...,$W$-$q_f$-...,$w$-$q$-...) is a $W$-$q$-space with respect to $\cP_1(X)$.

The games $\gm(AV(X,x,A,\cP_1(X)),I)$ and $\gm(AV(X,x,A,\cP_s(X)),I_f)$ coincide. Let us reformulate the game $\gm(AV(X,x,A,\cP_1(X)),I)$.

Let $X$ be a space, $A\subset X$, $x\in \cl A\setminus A$. Let us define the rules of the moves of the game $AV_1(X,x,A)$. Two players, $\al$ and $\be$, play.

\nstep
Player $\be$ chooses a neighborhood $V_n$ of the point $x$. Player $\al$ chooses $x_n\in A$.
\endnstep

\vskip 0.8\baselineskip

After a countable number of moves, we obtain the play:
\[
(V_0,x_0,...,V_n,x_n,...).
\]

Let us define the rule for determining the winner.
\begin{itemize}
\item[($I_1$)]
Player $\al$ wins if $\bigcap_\nom V_n\cap \cl{\set{x_n:\nom}}\neq\es$.
\end{itemize}

The game $\gm(AV(X,x,A,\cP_1(X)),I)$ is equivalent to the game $\gm(AV_1(X,x,A),I_1)$.

\begin{proposition}\label{p:gav:2}
Nearly $q_D$-spaces are $W_1$-$q_f$-spaces.
\end{proposition}
\begin{proof}
Let $X$ be an nearly $q_D$-space and $D\subset X=\cl D$ as in the definition of nearly $q_D$-spaces.
Let $f: X\to\R$ be a discontinuous function. By Proposition \ref{p:gav:1}(2), there exists $A\subset D$ and $x\in \cl A\setminus A$, such that $f(x)\notin \cl{f(A)}$. We will show that the game $\cG=\gm(AV(X,x,A,I_1)$ is $\al$-favorable. Let $s$ be a winning strategy for player $\al$ in the game $\cQ=\gm(NP(X,x,D),SQ)$. We will define a winning strategy for $\al$ in the game $\cG$. When constructing the strategy, we will also construct sequences $\sqnn U$ and $\sqnn W$ of open neighborhoods of the point $x$. We set $W_{n-1}=U_{n-1}=X$.

\nstep
Player $\be$ chooses a neighborhood $V_n$ of the point $X$.
Let $W_n$ be a neighborhood of the point $x$ such that $\cl{W_n}\subset W_{n-1}\cap V_n\cap U_{n-1}$.
Let $x_n\in W_n\cap A$. Let
\[
U_n=s(x_0,U_0,...,x_{n-1},U_{n-1},x_n).
\]
\endnstep

After a countable number of moves, we obtain a game $\sqn{V_n,x_n}$ of the game $\cG$, a game $\sqn{x_n,U_n}$ of the game $\cQ$, and a sequence $\sqnn W$. Since strategy $s$ is winning for $\al$ in the game $\cQ$, then $\sqnn x$ accumulates to some point $x'\in X$. Since $x_n\in W_n \subset \cl{W_n} \subset W_{n-1}\cap V_n$ for $\nom$, then $x'\in \bigcap_\nom V_n$. Therefore, condition $(I_1)$ is satisfied.
\end{proof}

\begin{proposition}\label{p:gav:3}
Countably pracompact spaces are $q_D$-spaces.
\end{proposition}
\begin{proof}
Let $D$ be a dense relatively countably compact subspace of $X$. Let $U_n=X$ for $\nom$. If $x_n\in U_n\cap D=D$ for $\nom$, then $\sqnn x$ accumulates to some point $X$.
\end{proof}

\begin{example}\label{e:gav:1}
Let $\Si$ be the $\Si$-product in $[0,1]^{\om_1}$ around zero, $L\subset [1/2,1]^{\om_1}$ be a subspace homeomorphic to the one-point Lindelification of a discrete space of cardinality $\om_1$ with a single non-isolated point $l_*$. Put $X=\Si\cup L$. Since $\Si$ is dense in $X$ and countably compact, $X$ is countably pracompact. Let $A=L\setminus \sset{l_*}$. Then the game $\gm(AV(X,l_*,A,\cP_s(X)),I)$ is $\be$-favorable, the winning strategy for $\be$ --- $V_n$ is chosen such that $V_n\cap S_i=\es$ for $i\leq n$. Therefore, $l_*$ is not a weakly $q$-point with respect to $A$.

The space $X$ is countably pracompact, but not a weakly $q$-space.
\end{example}

\begin{proposition}[{\cite[Example 1]{moors2013}}]\label{p:gav:3+1}
The product of \v Cech-complete spaces is a nearly $q_D$-space.
\end{proposition}

In \cite{AsanovVelichko1981} for $x\in \cl A\setminus A$ the notion of \term{weakly $q$-point with respect to $A$} is introduced, which coincides with the condition: the game $\gm(AV(X,x,A,\cP_s(X)),I)$ is $\al$-favorable.

A family of sets $\la$ is called \term{closed with respect to countable unions} if $\bigcup\mu\in\la$ for any at most countable $\mu\subset \la$.

\begin{proposition}\label{p:gav:4}
Let $X$ be a space, $Z\subset C_p(X)$, $\la$ be a closed cover of $X$ under countable unions such that for $M\in\la$ the following holds:
\begin{enumerate}
\item[{\rm ($\la_1$)}]
$S=\clx{C_p(X|M)}Z$ is compact, where $Z=\pi^X_M(Y)$ and
\item[{\rm ($\la_2$)}]
$S=\bigcup \set{ \cl Q : Q\subset Z,\ |Q|\leq\om }$.
\end{enumerate}
If $X$ is a $w$-$q_f$-space with respect to $\la$, then the closure of $Y$ in $C_p(X)$ is compact.
\end{proposition}
\begin{proof}
Assume the contrary. Let $K$ be the closure of $Y$ in $\R^X$. Since $\la$ is a closed covering of $X$ with respect to countable unions, it follows that $Y$ is pointwise bounded and, therefore, $K$ is compact. Let $f\in K\setminus C_p(X)$.
It follows from ($w$-$q_f$) that there exists a non-closed $A\subset X$ and $x\in \cl A\setminus A$, such that $f(x)\notin\cl{f(A)}$ and the game $\cG=\gm(AV(X,x,A,\la),I_f)$ is $\be$-unfavorable.
There exists $L_x\in \la$ such that $x\in\L_x$. Let $L_{-1}=\es$ and $V_{-1}=X$. Denote
\[
\Om(h,Q)=\sup \set{|h(x_1)-h(x_2)|: x_1,x_2\in Q}
\]
for $f\in \R^X$ and $Q\subset X$.

Let us define a strategy for $\be$ in the game $\cG$. At the $n$-th step, player $\be$ will also choose, in addition to the neighborhood $V_n$ of the point $x$, the sequence $\sq{f_{n,m}}{m\in\om}\subset Y$ and $L_n\in\la$.

\nstep
Since $\la$ is closed under countable unions, there exists $L_n\in\la$ such that
\[
L_x\cup L_{n-1} \cup \bigcup_{i<n} S_i \subset L_n.
\]
From ($\la_1$) it follows that $\pi^X_{L_n}(K)\subset C_p(X|L_n)$ is compact and from ($\la_2$) it follows that
\[
\restr f{L_n} \in \clx{C_p(X|L_n)}{\set{f_{n,m}:m\in\om}}
\]
for some $\sq{f_{n,m}}{m\in\om}\subset Y$. Take a neighborhood $V_n$ of $x$ such that $\cl{V_n}\subset V_{n-1}$ and
\[
\Om(f_{i,j},V_n)<\dfrac{1}{2^n}
\]
for $i,j\leq n$.
Player $\al$ chooses $S_n\subset A$ such that $S_n\subset L'$ for some $L'\in\la$.

\endnstep

Since the game $\cG$ is $\be$-unfavorable, there exists a game of $\cG$ in which player $\be$ used the described strategy, which was won by player $\al$, i.e. the condition ($I_f$) is satisfied: there exists
\[
S\subset G= \bigcap_\nom V_n
\]
such that $S\subset L$ for some $L\in \la$ and the set $S$ and $\wt S=\bigcup_\nom S_n$ are functionally inseparable.

Let $\cF=\set{f_{n,m}: n,m\in\om}$ and $L_*=\bigcup_{\nom}L_n$. Since $\restr f{L_n}\in \cl{\pi^X_{L_n}(\cF)}$ is for $\nom$, then $\restr f{L_*}\in \cl{\pi^X_{L_*}(\cF)}$. Since $\Om(f_{i,j},V_n)<\frac 1{2^n}$ for $n\geq \max i,j$, then $\Om(f_{i,j},G)=0$. Take \[ h\in \bigcap \set{ \clx K{\set{g\in\cF: \restr g{L_*}\in W}}: W \text{ is a neighborhood of }\restr f{L_*}\text{ in }\pi^X_{L_*}(K)}.
\]
Then $h\in \clx K\cF \subset K$ and $\restr h{L_*}=\restr f{L_*}$.
Since $\Om(g,G)=0$ for $g\in\cF$, then $\Om(h,G)=0$. Then from $x\in L_*$ it follows that $h(x)=f(x)$ and $h(G)=\sset{f(x)}$.
Since, by ($\la_1$), $M=L\cup L_*\in\la$, then $\restr hM\in \pi^X_M(K)\subset C_p(X|M)$. Take $q\in C_p(X)$ such that $\restr hM=\restr qM$. Note that $\restr h{L_*}=\restr q{L_*}$. Then
\begin{align*}
q(S)&\subset q(G)=h(G)=\sset{f(x)},
\\
q(\wt S)&\subset q(A\cap L_*)=f(A\cap L_*)\subset f(A).
\end{align*}
Since $f(x)\notin\cl{f(A)}$, the continuous function $q$ separates $S$ and $\wt S$. A contradiction with the fact that $S$ and $\wt S$ are functionally inseparable.
\end{proof}

The following assertion follows from Proposition \ref{p:gav:4}.

\begin{proposition}\label{p:gav:5}
Let $X$ be a space, $\cR$ be a boundedness-type property, $\la$ be a covering of $X$ closed under countable unions, such that for any $Y$ $\cR$-bounded in $C_p(X)$ and any $M\in\la$, the following holds:
\begin{enumerate}
\item[{\rm ($\la_1$)}]
$S=\clx{C_p(X|M)}Z$ is compact, where $Z=\pi^X_M(Y)$ and
\item[{\rm ($\la_2$)}]
$S=\bigcup \set{ \cl Q : Q\subset Z,\ |Q|\leq\om }$.
\end{enumerate}
If $X$ is a $w$-$q_f$-space with respect to $\la$, then $X$ is a \fmrus \cR/-complete space.
\end{proposition}

Proposition \ref{p:gav:5} is a generalization of \cite[Proposition 3.4]{AlperinOsipov2023}.

We call a space $X$
\term{($w_b$-$q_f$-...,$W_b$-$q_f$-...,$w_b$-$q$-...) a $W_b$-$q$-space} if $X$ ($w$-$q_f$-...,$W$-$q_f$-...,$w$-$q$-...) is a $W$-$q$-space with respect to $\cP_b(X)$.

\begin{theorem}\label{t:gav:1}
$w_b$-$q_f$-Spaces are \fmus cc/-complete.
\end{theorem}
\begin{proof}
Let $X$ be a $w_b$-$q_f$-space.
The family $\la=\cP_b(X)$ is a closed covering of $X$ with respect to countable unions.
Let $Z$ be a relatively countably compact subset of $C_p(X)$.
Proposition \ref{p:rbs:3} implies that the closure of $\pi^X_Y(Z)$ in $C_p(Y|X)$ is an Eberlein compactum for any $Y\in\la$ - condition ($\la_1$) of Proposition \ref{p:gav:5} is satisfied. Since the tightness of Eberlein compacta is countable, the condition ($\la_2$) of Proposition \ref{p:gav:5} holds. Proposition \ref{p:gav:5} implies that $X$ is \fmus cc/-complete.
\end{proof}

Proposition \ref{p:gav:5} implies the following assertion.

\begin{proposition}\label{p:gav:6}
Let $X$ be a space, $\cR$ be a boundedness-type property, $\la$ be a closed covering of $X$ with respect to countable unions, such that the following conditions hold:
if $M\in\la$, then $M$ is \hfmrus \cR/-complete and every compact subspace of $C_p(M)$ has countable tightness. If $X$ is a $w$-$q_f$-space with respect to $\la$, then $X$ is \fmrus \cR/-complete.
\end{proposition}

\begin{theorem}\label{t:gav:2}
$w_s$-$q_f$-Spaces are \fmus b/-complete.
\end{theorem}
\begin{proof}
Let $X$ be a $w_s$-$q_f$-space.
The family $\la=\cP_s(X)$ is a closed covering of $X$ under countable unions.
Let $Z$ be a bounded subset of $C_p(X)$.
Let $Y\in\cP_s(X)$. Since $Y$ is separable, $C_p(Y)$ is submetrizable and hereditarily \mus b/-complete.
Therefore, the closure of $\pi^X_Y(Z)$ in $C_p(Y|X)$ is a metrizable compactum. Proposition \ref{p:gav:5} implies that $X$ is \fmus b/-complete.
\end{proof}

\begin{proposition}\label{p:gav:7}
Let $X$ be a space. Then {\rm (1) $\larr$ (2) $\larr$ (3) $\larr$ (4) $\larr$ (5)}:
\begin{enumerate}
\item
$X$ is \fmus b/-complete;
\item
$X$ is a $w_s$-$q_f$-space;
\item
$X$ is a $W_1$-$q_f$-space;
\item
$X$ is an nearly $q_D$-space;
\item
$X$ is a product of \v Cech-complete spaces.
\end{enumerate}
\end{proposition}
\begin{proof}
\noindent
(1) $\larr$ (2) Theorem \ref{t:gav:2}.

\noindent
(2) $\larr$ (3) Obviously.

\noindent
(3) $\larr$ (4) Proposition \ref{p:gav:2}.

\noindent
(4) $\larr$ (5) Proposition \ref{p:gav:3+1}.
\end{proof}

We call a space $X$
\term{($w_c$-$q_f$-...,$W_c$-$q_f$-...,$w_c$-$q$-...) a $W_c$-$q$-space} if $X$ ($w$-$q_f$-...,$W$-$q_f$-...,$w$-$q$-...) is a $W$-$q$-space with respect to $\cP_c(X)$.

\begin{theorem}\label{t:gav:3}
$w_c$-$q_f$-spaces are \fmus pc/-complete.
\end{theorem}
\begin{proof}
Let $X$ be a $w_c$-$q_f$-space.
The family $\la=\cP_c(X)$ is a closed covering of $X$ under countable unions.
Let $Z$ be a pseudocompact subset of $C_p(X)$.
Proposition \ref{p:rbs:3} implies that the closure of $\pi^X_Y(Z)$ in $C_p(Y|X)$ is an Eberlein compactum for any $Y\in\la$ --- condition ($\la_1$) of Proposition \ref{p:gav:5} holds. Since the tightness of Eberlein compacta is countable, condition ($\la_2$) of Proposition \ref{p:gav:5} holds. Proposition \ref{p:gav:5} implies that $X$ is \fmus pc/-complete.
\end{proof}


\end{fulltext}

\bibliographystyle{abbrv}

\end{document}